\documentclass[11pt]{article}

\usepackage{amsfonts,amssymb,latexsym}
\usepackage{amsthm}
\usepackage[margin=1in]{geometry}
\usepackage{graphicx}
\usepackage{amsmath}
\usepackage[font=footnotesize,labelfont=bf]{caption}
\usepackage{subfig}
\usepackage{placeins}
\usepackage{tikz}
\usetikzlibrary{calc,intersections}

\begin{document}

\numberwithin{figure}{section}
\newtheorem{thm1}{Theorem}[section]
\newtheorem{thm}[thm1]{Theorem}
\newtheorem{lem}[thm1]{Lemma}
\newtheorem*{thm2}{Theorem}
\newtheorem*{definition}{Definition}
\newtheorem{cor}[thm1]{Corollary}
\newtheorem*{claim}{Claim}
\newtheorem*{conj}{Conjecture}
\newtheorem*{rem}{Remark}
\newtheorem*{question}{Question}
\newtheorem{ex}{Example}[section]

\newtheorem{corollary}[thm1]{Corollary}
\newtheorem{conjecture}[thm1]{Conjecture}
\newtheorem{theorem}[thm1]{Theorem}
\newtheorem{lemma}[thm1]{Lemma}
\newtheorem{exer}[thm1]{Exercise}
\newtheorem{proposition}[thm1]{Proposition}

\newcommand{\bi}{\begin{itemize}}
\newcommand{\ei}{\end{itemize}}
\newcommand{\be}{\begin{enumerate}}
\newcommand{\ee}{\end{enumerate}}
\newcommand{\ds}{\displaystyle}
\newcommand{\ul}{\underline}

\interfootnotelinepenalty=10000

\title{Crossing Number Bounds in Knot Mosaics}
\date{\today}
\author{Hugh Howards, Andrew Kobin}

\maketitle

\abstract


\noindent Knot mosaics are used to model physical quantum states. The mosaic number of a knot is the smallest integer $m$ such that the knot can be represented as a knot $m$-mosaic. In this paper we establish an upper bound for the crossing number of a knot in terms of the mosaic number. Given an $m$-mosaic and any knot $K$ that is represented on the mosaic, its crossing number $c$ is bounded above by $(m - 2)^{2} - 2$ if $m$ is odd, and by $(m - 2)^{2} - (m - 3)$ if $m$ is even.  In the process we develop a useful new tool called the mosaic  complement.

\section{Introduction}

In \cite{lk}, Lomonaco and Kauffman introduce a standard system of knot mosaics as a model of physical quantum states. Since then mosaics have been used in this original manner, but also as a tool by which to better understand knots themselves.

Much like crossing number and other easy to define invariants of knots it is not hard to get very loose bounds for mosaic number, but strong bounds can be difficult to compute.   Many of the recent known results can be found at the Online Encyclopedia of Integer Sequences webpage devoted to the number of $n \times n$ knot mosaics \cite{slo}.

Mosaics contain 5 distinct tiles, up to rotation, and all 11 orientations are shown below.  We label the tiles with roman numerals for the 5 types and when applicable the letters $a$ though $d$ for the distinct rotations of those types. We also introduce a type 0 tile which consists of a square with a dot in the center.  Type 0 tiles are not a part of a mosaic, and have not previously been used in the literature but will be used when we introduce a new tool called the mosaic  complement in Section~\ref{sec:mosaic  complementdef}.

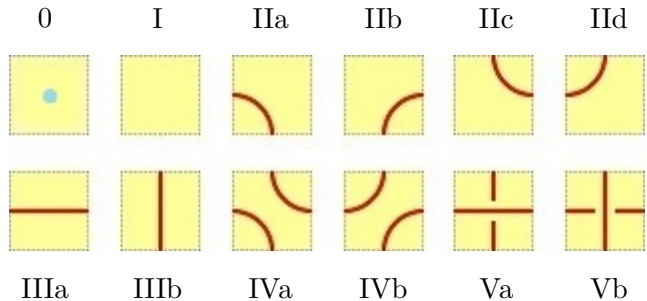
\begin{figure}[tpb]

\begin{center}
\begin{tikzpicture}
  \node at (0,0) {\includegraphics{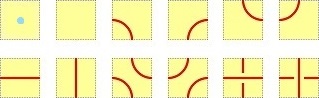}};
  \node at (-3.75,1.8) {0};
  \node at (-2.25,1.8) {I};
  \node at (-.75,1.8) {IIa};
  \node at (.75,1.8) {IIb};
  \node at (2.25,1.8) {IIc};
  \node at (3.75,1.8) {IId};
  \node at (-3.75,-1.8) {IIIa};
  \node at (-2.25,-1.8) {IIIb};
  \node at (-.75,-1.8) {IVa};
  \node at (.75,-1.8) {IVb};
  \node at (2.25,-1.8) {Va};
  \node at (3.75,-1.8) {Vb};
\end{tikzpicture}
\end{center}

\caption{The tiles of a mosaic and its mosaic  complement.}
\label{fig:tiles}
\end{figure}

For a positive integer $n$, define an \emph{$n$-mosaic} $M_n$ as an $n\times n$ matrix composed of mosaic tiles. As is typical when studying mosaics, we are only interested in mosaics which represent knots and links so we restrict to those cases in the standard way as follows.
A {\em connection point} is the midpoint of a tile edge which is also
the endpoint of an arc drawn on that tile.  Thus type I tiles have no connection points,
type II and III tiles have 2 connection points and type IV and V tiles have 4 connection points
as seen in Figure~\ref{fig:tiles}.
 {\em Contiguous} tiles are any tiles which fall directly next to each
other in the same column or row. We can then say that a tile within a mosaic
is {\em suitably connected}
if each of its connection points is identified with a connection
point of a contiguous tile.


A \emph{link $n$-mosaic} is an $n$-mosaic with all of its tiles suitably connected so that after all the tiles are placed on the mosaic, 
the result is a projection of a link.  In such a mosaic, tiles of types II, III, IV, and V must have adjacent tiles that are also of one of these four types to extend the arcs started on those tiles.   

A \emph{knot $n$-mosaic} is a link $n$-mosaic that
corresponds to a projection of a one component link (a knot).
Thus every knot mosaic is a link mosaic, but not every
link mosaic is a knot mosaic.
 Define {\em the mosaic number} of a knot (or link) to be the smallest natural number $m$ such that the knot (link) is able 
 to be represented as a knot (link) $m$-mosaic.  See Figure~\ref{fig:trefoil} for an example of a 4-mosaic that is not suitably connected as well as a knot 4-mosaic (that is suitably connected).

\begin{figure}[tpb]
\centering
\includegraphics[scale=.9]{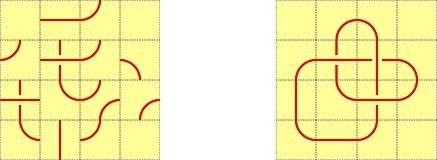}
\caption{Pictured above are two $4\times 4$ mosaics, but only the one on the right is suitably connected
to yield the projection of a knot.}
\label{fig:mosaicex}
\label{fig:trefoil}
\end{figure}

An important motivation for studying knot mosaics is the Lomonaco-Kauffman Conjecture, which states that knot mosaic theory is equivalent to classic (tame) knot theory. This was proven by Kuriya and Shehab in \cite{kuriya}, so as a result, we can treat knot equivalence and crossing number of knot mosaics in the usual way.

At times we will want to examine specific tiles within a mosaic.  
When specifying a given tile, we model subscripts after the entries in a matrix so the tile $R_{i,j}$ will refer to the tile in the $i^{th}$ row and
$j^{th}$ column, where columns are counted from the left and rows are counted {\bf from the bottom} (we diverge from matrix notation slightly here to allow row numbers to reflect a height function).
If we think of $M_n$ as an $n \times n$ square disk, the $4n-4$ tiles that intersect $\partial M_n$ will be called \emph{boundary tiles} and the other $(n-2)^2$ tiles will be referred to as the \emph{interior} of the mosaic and denoted $S$.  See Figure~\ref{fig:s} for a depiction of $S$ and the boundary tiles in a mosaic.  We will often be focused on $S$.

\begin{figure}[htp]
\centering
\includegraphics[scale=.9]{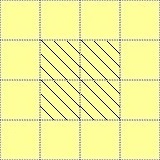}
\caption{In a $4\times 4$ mosaic the the interior $S$ is the shaded $2\times 2$ sub-mosaic.}
\label{fig:s}
\end{figure}

In this paper we are explore the relationship between crossing number and mosaic number.  We begin by listing some of the known results on crossing number and mosaics.  We start with Theorem 3.1 in \cite{gard}.

\begin{thm}[Upper Bound on Crossing Number] 
\label{thm:lowerbound} \cite{gard}
Given an $m$-mosaic, if a knot is representable on the mosaic then its crossing number is bounded above by: $c \leq (m - 2)^{2}$
\end{thm}

This upper bound follows immediately from the well known observation that all crossing tiles on an $n$-mosaic occur on the interior $S$ of a mosaic (see Lemma~\ref{lem:cross}) and $S$ has exactly $(n - 2)^{2}$ tiles. 
Using this bound, one can quickly find a lower bound for the mosaic numbers for many knots.
For a few simple knots such as the trefoil it yields the exact mosaic number.
 However, it is clear that this bound is of limited use when it comes to determining the mosaic number of more complex knots.
Ludwig, Evans, and Paat, for example, show that the relationship between crossing number and mosaic number can be subtle in \cite{lep} by building an infinite family of knots where each knot can only achieve its mosaic number in a projection on a mosaic that does not realize its crossing number.

While we will focus on an upper bound, Lee, Hong, Lee, and Oh give a lower bound on crossing number  in \cite{lee}  showing that the mosaic number of a non-trivial knot is always less than or equal to the crossing number of a knot plus 1.   In this paper, we sharpen the upper bound on crossing number by proving

\bigskip
\noindent {\bf Theorem~\ref{thm:newub}}.
Given an $m$-mosaic and any knot $K$ that is projected onto the mosaic, the crossing number $c$ of $K$ is bounded above by the following:\\
\begin{equation*}
  c \leq
  \begin{cases}
    (m - 2)^{2} - 2 & \quad \text{if $m = 2k + 1$}\\
    (m - 2)^{2} - (m - 3) & \quad \text{if $m = 2k$.}
  \end{cases}
\end{equation*}

In the next section we introduce a new tool, the {\em mosaic  complement}, denoted $C$, together with $T$, an ordered triple, associated to $C$.

\section{The definition of the mosaic  complement $C$ and $T$ an ordered triple associated to $C$}

\label{sec:mosaic  complementdef}

Given a mosaic $M_n$ for link $L$ (or knot $K$), we define the mosaic  complement on $S$, the interior tiles of $M_n$.
The mosaic  complement $C$ is obtained by replacing the tiles on $S$ in the following manner:
every type V  is replaced with a type I, type IV with type 0, type IIIa with type IIIb, type IIIb with type IIIa, type IIa with type IIc, type IIc with type IIa, type IIb with type IId, type IId with type IIb, and finally type I with either a type IVa or IVb (although this means the mosaic  complement is not uniquely defined, either type IVa or IVb is fine).

Intuitively, the mosaic  complement is complementary to the link in the mosaic in the sense that the union of the mosaic  complement and the mosaic will intersect the boundary of each tile in $S$ exactly four times - once on each of the tile's edges.   
The name follows from the fact that if we put a link on a mosaic together with its mosaic complement, the union of the two types of arcs on the interior of the mosaic $S$ will consist of entirely type V tiles, type IV tiles, and type IV tiles together with a dot (the last set comes from when the link has a type IV tile and the mosaic complement contributes a type 0 tile).  One set of arcs on these tiles belongs exclusively to the link and the complementary set of arcs belongs only to the mosaic complement.
%

Although the definition of the mosaic  complement is not quite unique since we could choose either type IV tile to replace each type I tile, there is an inverse function associated to the definition that is unique (reverse the orders in the definition above), so
the portion of the link in the mosaic contained in $S$ may be deduced from $C$.  
See Figure~\ref{fig:basicdual} for a picture of a knot mosaic together with its mosaic  complement (the outline of $S$ is also pictured).

In general, the mosaic  complement consists of loops, type 0 tiles, and arcs with both endpoints on
$\partial S$ which we call edges.  The term arc will be used to refer to a subset of a loop or an edge.  {\em The constant $|C|$} is defined to be
the total number of tiles in $C$ (excluding the blank type I tiles, of course).   {\bf Throughout the paper $R_{i,j}$ refers to the tile representing the knot and $T_{i,j}$  refers to the corresponding tile for the mosaic  complement.}  Although  $T_{i,j}$ is defined by looking at $R_{i,j}$, we will be focused on the mosaic  complement in 
most of our arguments so we will almost always refer directly to $T_{i,j}$.

\begin{figure}[htp]
\centering

 {\includegraphics[scale=.7]{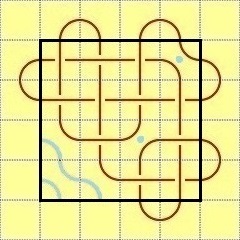}}
\caption{An example of a knot mosaic and its mosaic  complement.}
\label{fig:basicdual}
\end{figure}

We next define an ordered triple $T$, which is computed from the mosaic  complement.  Recall that $C$ is the set of all tiles (that are not blank) in the mosaic complement.
Let $C'$ be the set of all tiles that form the complement in $C$ of all of the type 0 tiles in $C$.  Note that the tiles of $C'$  together form a set of loops and properly embedded edges on the square disk $S$. 
Let $C''$ be the set of  all tiles $T_{i,j}$ in $C$ that are of type IV. 
Let $l = |C|$, $l' = |C'|$ and $l''=|C''|$.   We define the ordered triple $T= (l,l',l'')$. 
Notice that although the mosaic  complement is not unique for a given mosaic, the only choices were which of the two type IV tiles to pick and so any choice of mosaic  complement for a given mosaic will give the same ordered triple $T$.  In Figure~\ref{fig:basicdual}, for example, $T=(5,3,1)$ because the mosaic complement has 
5 nontrivial tiles, 3 of those tiles are not type 0, and 1 of those tiles is type IV.
In Figure~\ref{fig:m6} , $T=(3,0,0)$ because there are 3 nontrivial tiles in the mosaic complement,
but 0 of those tiles are not type 0 and, of course, 0 of those tiles are type IV. 

Of all possible ways to embed a specific knot $K$ on an $n \times n$ mosaic, let $M_n$ be a mosaic that minimizes the ordered triple $T$ lexicographically,
and say that in such a case that $T$ is {\em minimal}.  For example, if $K$ can be built with with a mosaic  complement yielding $T_1=(7,4,2)$ or with a mosaic  complement yielding $T_2 = (8,0,0)$ we pick the first embedding since $T_1 < T_2$ lexicographically.  



\section{Saturation and mosaic  complements}

\label{sec:saturation}

 A mosaic is said to be \emph{saturated}  if every tile of $S$, the interior of the mosaic, is a crossing tile. In this case $C = \emptyset$.
Conversely, if a link mosaic has a nonempty mosaic  complement, it is not saturated.
Theorem~\ref{thm:newub} stated above shows that the even and odd knot mosaic boards have radically different properties regarding how ``nearly saturated'' they can be.

\begin{figure}[htp]
\centering
\includegraphics[scale=.6]{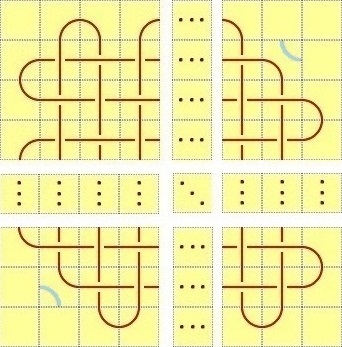}
\caption{An odd mosaic with $|C| = 2$.}
\label{fig:endlessgeneral}
\end{figure}

\begin{figure}[htp]
\centering
\includegraphics[scale=.7]{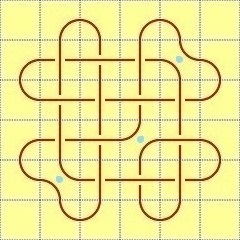}
\caption{There are three type 0 tiles in this mosaic  complement on $M_{6}$.  One could think of the 
mosaic as a saturated mosaic with three crossings smoothed reducing the number of components in the link from
4 to 1.}
\label{fig:m6}
\end{figure}

\begin{lem}
In a link mosaic, boundary tiles cannot be crossing tiles. Therefore all crossings of a link must fit on $S$, the interior of the mosaic.
\label{lem:cross}
\end{lem}

\begin{proof}  This Lemma is easy to prove and well known.
By definition for a mosaic to be suitably connected, connection points cannot occur on the boundary of the mosaic.  Every edge of a crossing tile contains a connection point, so therefore boundary tiles cannot be crossing tiles (Tile $R_{1,2}$  in the mosaic on the left in  Fig.~\ref{fig:mosaicex}.
\end{proof}

\begin{figure}[htp]
\centering
\includegraphics[scale=.9]{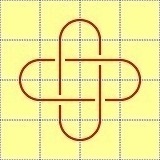}
\caption{A projection of Solomon's link on a saturated mosaic ($M_{4}$).}
\end{figure}

\begin{thm}
\label{thm:saturated}
A saturated mosaic cannot contain the projection of a knot that achieves its crossing number.
\end{thm}

\begin{proof}
Start by filling $S$ with type V tiles.  The proof will not depend on if we choose type Va or Vb so at this stage we have not lost generality of the argument no matter which type V tiles we choose.  Now we notice that since edges intersecting $\partial S$ can only connect to one of its two adjacent edges in $\partial S$
 there are only two choices of how to connect up the strands through the boundary tiles
to get a knot or link.
The vertical strand in $R_{2,3}$, for example, must either connect to the vertical strand in $R_{2,2}$ as it does in Figure~\ref{fig:evenlink}  or $R_{2,4}$
as in Figure~\ref{fig:westwestloops}.  In the first case this means tile $R_{1,3}$ is a type IId and $R_{1,2}$ is type IIc and in the second case tile $R_{1,3}$ is a type IIc and $R_{1,4}$ is a type IId, again both matching the examples in Figures~\ref{fig:evenlink} and~\ref{fig:westwestloops} respectively.  On a saturated board once this single choice has been made the rest of the choices on the outside are uniquely determined.

Suppose $K$ is a mosaic representation on $M_{n}$ which is saturated.  If $n$ is odd, both choices of how to connect up along the boundary tiles leads to a pair of nugatory crossings in opposite corners of the board like the crossing shown in Fig.~\ref{fig:loop}.  If $n$ is even, one choice will result in a link (not a knot) and the other will result in a knot with a nugatory crossing in all four corners.  
In each of the cases where we have a knot instead of a link type I Reidemeister moves will reduce the number of crossings.  Therefore a saturated mosaic cannot contain the projection of a knot that achieves its crossing number.
\end{proof}

\begin{figure}[htp]
\centering
 \reflectbox{\rotatebox[origin=c]{180}{\includegraphics[scale=.9]{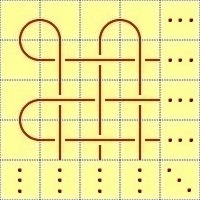}}}
\caption{A nugatory crossing in the corner.}

\label{fig:westwestloops}
\label{fig:loop}
\end{figure}

\begin{figure}[htp]
\centering
\includegraphics[scale=.6]{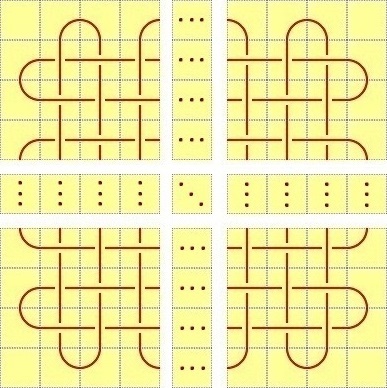}
\caption{A link of $n - 2$ components.}
\label{fig:evenlink}
\end{figure}

\begin{figure}[htp]
\centering
\includegraphics[scale=.6]{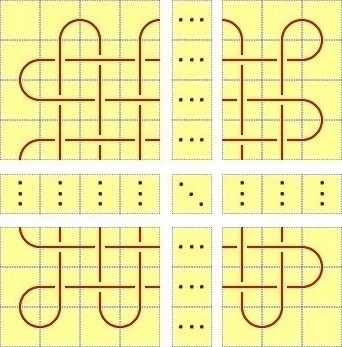}
\caption{Odd mosaic with corners easily removed by type I Reidemeister moves..}
\label{fig:oddcorners}
\end{figure}

We next consider mosaics with mosaic  complements consisting of a single tile, that is $|C| = 1$. Almost every knot will fail to achieve its crossing number on such mosaics, and the following lemma precedes a general theorem for mosaics with $|C| = 1$.

\begin{lem}
The trefoil knot has mosaic number 4.
\end{lem}

\begin{proof}
The trefoil knot has crossing number 3. 
Since a $3\times 3$ board can only support one crossing, we must have at least a 
$4\times 4$ board to have a non-trivial knot.  
Indeed, the mosaic on the right in Figure~\ref{fig:trefoil} shows that the trefoil knot may be embedded on $M_{4}$ and achieve its crossing number of 3. Therefore the trefoil knot has mosaic number 4.
\end{proof}


\begin{thm}
\label{thm:oddub}
Given a knot K with crossing number $c$, suppose its mosaic number $m$ is odd. Then $c \leq (m - 2)^{2} - 2$.
\end{thm}

\begin{proof}

We showed above that an odd mosaic represents a knot with crossing number at most $(n-2)^2 - 1$.
A mosaic with two interior tiles that are not crossing tiles (type V) will have $c \leq (m - 2)^{2} - 2$ so we are left to focus on the case of exactly one interior tile that is not a type V tile.
If the knot $K$ is achieved with only one interior tile that is not a crossing tile, then this means $|C|=1$.  This can only happen if $C$ is either a single type 0 tile or if $C$ is a single type II  tile in one of the corners of $S$.  If $C$ is not in one of the corners of $S$, then $M_n$ has two opposite corners with crossings that can be reduced by a type I Reidemeister move just as in a saturated odd board.  Placing the mosaic  complement in one of the corners can at most remove one of these nugatory crossings.  Thus even though this embedding of $K$ has $(n-2)^2-1$ crossings, $K$ does have an embedding with $(n-2)^2 - 2$ crossings or less, bounding the crossing number from above.
\end{proof}

Theorem \ref{thm:oddub} establishes the first part of the main theorem in this paper (Theorem \ref{thm:newub}).
Moreover, an odd mosaic that is saturated and with the type V crossings chosen to give an alternating knot achieves this bound, so the bound is sharp.
See Figure~\ref{fig:endlessgeneral} for the same knot simplified via two type I Reidemeister moves to show that the knot achieves its crossing number.

Once the even case is established, it will follow that the trefoil is the only knot which can be constructed on a mosaic whose mosaic  complement consists of a single tile.
For the rest of the paper we focus on the even mosaic case $M_{n}$ where we assume $n$ is even and $T$ is minimal with respect to all $n \times n$
knot mosaics giving knot $K$.

\section{Loops in the mosaic  complement}

We now construct an argument showing that we can assume $C$ contains no loops while keeping $T$ minimal.  

\begin{lemma}
Let $M_n$ be a mosaic for a knot $K$ for which $T$ is minimal and the number of loops in mosaic  complement $C$ is minimized over all such possible mosaics and for which $|C| \leq n-4$.  Let
$\{c_1,c_2,\dots,c_k\}$ be the set of loops in the mosaic  complement.
Then if the set of loops is not empty, at least one of the $c_i's$ contains a type II tile.

\label{lem:loop1}
\end{lemma}

\begin{proof}

Since each loop has at least 4 corners, the only way for a loop to avoid a type II tile is if each corner is part of a type IV tile.  
If none of the corners are type II then the loop has at least 4 type IV corners that meet other components of the mosaic  complement.  If one of these corners is type IVa replace it with one that is type IVb; if not do the opposite.  This swap yields the connect sum of the loop in the mosaic  complement with another component of the mosaic  complement, decreasing the number of loops in the mosaic  complement by one.  Since it does not change $T$ and since it still yields a mosaic  complement for $K$  (the type IV corners are chosen arbitrarily) we see that the original mosaic  complement did not meet the hypothesis of the lemma, a contradiction.  Thus the corners of each loop may be assumed to be type II, an even stronger conclusion than the one type II tile in the lemma. 
\end{proof}

\begin{lemma}
If some loop in the mosaic  complement contains a type II tile then $T=(l,l',l'')$ is not minimal.
\label{lem:loop2}
\end{lemma}

\begin{proof}
If there is such a loop, call it $c_1$.
 Any time $c_1$ crosses the knot, $K$, we may dictate that it goes under $K$. By virtue of the definition of the mosaic  complement, $c_1$ never crosses itself.  Now remove $c_1$ from the mosaic  complement and add it instead to the knot, 
replacing the knot mosaic with a link mosaic containing $K$ and an unknot.  Next place a type V tile into the mosaic (type V in the link, not the mosaic  complement) where the type II tile of $c_1$ had been. 
 Because $c_1$ was entirely below $K$, and $c_1$ had no crossings with itself, we have just taken the connect sum of $K$ with an unknot.  Thus we have a new mosaic representing $K$, but $|C|$ has dropped contradicting the minimality of the ordered triple
$T$ in the original mosaic.
\end{proof}

In the proof above we found an unknot in the mosaic  complement that contained a type II tile, swapped it out of the mosaic  complement and into the mosaic and changed the type II tile to a type V, yielding a connect sum of $K$ with an unknot resulting in a new embedding of $K$ and lowering $T$.  We will repeat this process multiple times in different contexts throughout the paper and we call the process {\em corner conversion}.

Lemmas~\ref{lem:loop1} and~\ref{lem:loop2} imply
\begin{cor}
Let $M_n$ be a mosaic for knot $K$ with $|C| \leq n-4$ and for which $T$ achieves its minimum over all such mosaics.  
We may then choose $M_n$ so that simultaneously $C$ contains no loops and $T$ is minimal.

\label{cor:noloops}

\end{cor}

We now want to look at a particular class of loops.  These are the shortest possible loops: ones of length 4
coming from a combination of 4 tiles of types II and IV.  We call these short loops {\em bubbles}.  We see a bubble in each of the pictures in Figure~\ref{fig:percolate}.

The proof above showed that the mosaic  complement may be chosen to contain no loops if $T$ is minimal,
but it did not show that a mosaic  complement containing loops could not also be minimal.   Later 
we may start with a mosaic  complement that contains no loops and use moves that create bubbles, which we then want to argue is impossible, so we need a stronger result saying that if $T$ is minimal the mosaic  complement cannot contain bubbles.
In the argument it will be useful to have the following lemma that gives us some flexibility in where a bubble might be positioned.

\begin{lemma}
Given a knot mosaic $M_n$ for knot $K$ with mosaic  complement $C$ and ordered triple $T$, if $M_n$ contains any 
$2 \times n$ or $n \times 2$ subset of tiles, $n \geq 2$, that 
consists of exclusively type IV tiles in the complement, then we may pick 
any $2 \times 2$ subset of these tiles and form 
a (possibly) new mosaic  complement for $M_n$ and $K$ in which there is a bubble
in the $2 \times 2$ subset so that $T$ is unchanged for the new mosaic  complement.  
\label{lem:percolate}
\end{lemma}

\begin{proof}
The proof is easy.  Pick each of the four tiles to be IVa or IVb so that they have a bubble within them.  Leave the other type IV tiles unchanged.  Since swapping type IV tiles in the mosaic  complement doesn't affect $K$ or $T$, the new mosaic  complement has shifted the bubble to the desired location and satisfies our requirements on $K$ and $T$. 
\end{proof}

\begin{figure}[htp]
\centering
\includegraphics[scale=.9]{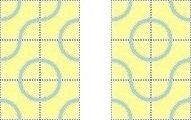}
\caption{A bubble percolates up through a $3 \times 2$ set of type IV tiles.}
\label{fig:percolate}
\end{figure}

This process allows us to shift the location of an existing bubble through nearby type IV tiles. We call this process of shifting a bubble to a new location {\it percolation}. A $3 \times 2$ example is shown in Figure~\ref{fig:percolate}.  

\begin{theorem}
Let $M_n$ be a mosaic for knot $K$ for which $T$ is minimal.   $C$ cannot contain a bubble.
\label{thm:burst}
\end{theorem}

\begin{proof}
If we ever have a bubble which contains a type II tile, then we can 
do a corner conversion as we did in Lemma~\ref{lem:loop2},
 yielding a new embedding of $K$ but lowering $T$.  This contradicts the fact that $T$ was minimal.  Thus we proceed with the argument under the assumption that the bubble is entirely contained in type IV tiles.

Next we show the mosaic  complement is not minimal if the knot intersects either a row or a column of $S$ containing the bubble.  
In this case, the ability to rotate the mosaic assures us that we may assume that the intersection is in a column above the bubble. If there are any type IV tiles below the knot in those columns, but above the bubble we use Lemma~\ref{lem:percolate} to shift the bubble up to the four tiles directly below the knot.

Explicitly, if $K$ intersects $T_{i+1,s} \cup T_{i+1,s+1}$ we let the bubble be contained in tiles $T_{i-1,s}$, $T_{i,s}$, $T_{i-1,s+1}$, and $T_{i,s+1}$.  Without loss of generality let $K$ intersect $T_{i+1,s}$ (and possibly $T_{i+1,s+1}$, too).
Because it is directly above a type IV tile, but it contains part of the knot,
$T_{i+1,s}$ is either type IIb, IIa, or IIIb (the knot cannot intersect the bottom edge of the tile). $T_{i+1,s+1}$
is also above a type IV tile and must pair with $T_{i+1,s}$.
This means $T_{i + 1,s + 1}$ must be a type IIa tile or a type IV tile if $T_{i+1,s}$ is type IIb, and $T_{i+1,s+1}$ must be IIb tile or a type IIIb
if $T_{i+1,s}$ is IIa or IIIb.

We show moves in Figure~\ref{fig:burst} for six possible combinations that allow us to connect sum the bubble with the knot and reduce $T$ contradicting minimality -- in the case of a type IV in $T_{i+1,s+1}$ we show only  type IVb case since the IVa case is nearly identical.

\begin{figure}[htp]
\centering
\includegraphics[scale=.8]{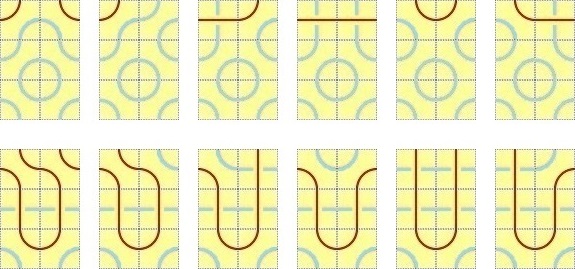}
\caption{A bubble can never appear in the mosaic  complement when $T$ is minimal.  Here we see that if a tile above the bubble contains an arc of $K$, there is always an embedding of $K$ that decreases $T$ and `bursts' the bubble.}
\label{fig:burst}
\end{figure}

Finally we are left with the case where neither the columns nor the rows of $S$ containing the bubble intersect the knot.  Thus they are exclusively full of type IV tiles in the complement.  Since $K$ exists and in any interesting case has at least one crossing, we know that some row of $S$ intersects $K$.
(Of course the unknot fits on a $2 \times 2$ board with $S = \emptyset$ so we are only interested in knots with positive crossing number.)

By Lemma~\ref{lem:percolate} we can percolate the bubble within the columns containing it to make it intersect the row that already contains part of $K$.
Now as before we have not changed $K$ or $T$.  We have, however, reduced to the previous case, which shows $T$ can be reduced without changing $K$, contradicting minimality.
\end{proof}

\section{Edges in the mosaic  complement}

Since we now know that we can get rid of loops in a mosaic  complement without increasing $T$ we turn our attention to edges.

\begin{lemma}
If $|C| \leq n-4$ then no edge in $C$ runs from one side $S$ to the opposite side.
\end{lemma}

\begin{proof}
The proof is trivial since such an edge must be of length at least $n-2$.
\end{proof}

This implies that if $e$ is an edge of the mosaic  complement, then the endpoints will have to either be on adjacent sides, as in 
edges $f_1, f_2, f_3$ and $f_4 $  in Figure~\ref{fig:outermost}, or on the same side of $S$, i.e. starting and ending in the exact same row or column as in edges $e_1, e_2$ and $e_3$ in the same figure.
If the endpoints are on the same side as each other we call it an {\em XX-edge}
and if adjacent sides we call it an {\em XY-edge}.

Both XX and XY-edges cut $S$ into two disks.  
The smaller side is considered the outside of the edge.
In topology we may not have a metric so we often avoid talking about the smaller part of a disk, but a mosaic as an $n \times n$ subset of the plane has a natural metric on it so we are free to use the term.

Because the edges in the mosaic  complement are relatively short, if $e$ is an XX-edge then the boundary of the outside (smaller) disk consists of $e$ together with part of one side of $S$.
Likewise if $e$ is an XY-edge the outside consists of $e$ together with parts of two sides of $S$.
Thus for any arc $e$ we have a notion of \emph{outside}.
An edge $e \subset C'$ is called \emph{outermost} in $C'$ if there are no edges outside of $e$ on $S$ in $C'$.  

Note that if $C$ contains no loops -- which Corollary~\ref{cor:noloops} allows us to assume -- and $C' \neq \emptyset$, so there is at least one edge, then there must be an outermost edge $e$.  
Also note that our definition is not quite the same as the traditional definition of outermost arcs on disks in topology. 
If $e$ is outermost in our context it is outermost in the traditional definition, but not every traditionally outermost arc is outermost in our definition because it might be outermost, but on the wrong side (the side of its larger disk).  
An edge can still be outermost even if there is a type 0 tile of $C$ outside of it.

\section{Reduction moves}

We establish a set of moves that when applied to the arcs of the mosaic  complement will reduce the ordered triple $T$ without changing the isotopy class of the knot.  One primary use of the moves is to lower an arc $a$ of the mosaic  complement that represents a local maximum (possibly after rotating the mosaic).
This will eventually lead to the conclusion first that no such moves can be made to the edges of $C$, and then after a further argument that $C$ contains no edges at all if $T$ is minimal.  We define the moves starting with the more elementary moves.

\subsection{The type IV moves: bubble release and XX-through-XY moves}

In the proof of Lemma~\ref{lem:loop1} we swapped type IV tiles to reduce the number of loops in a mosaic  complement by taking the connect sum of a loop with another part of the mosaic  complement.  We now consider the inverse operation on the mosaic  complement when it would create a bubble.  We define a bubble release move when we swap a type IVa tile for IVb or vice versa to yield a bubble without altering $K$ or $T$.  Such a move is pictured in Figure~\ref{fig:bubblerelease}.

\begin{figure}[htp]
\centering
\includegraphics[scale=.9]{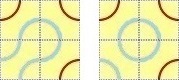}
\caption{Neither the embedding of the knot nor $T$ are altered when we break a bubble off of an arc 
of the mosaic  complement using the bubble release move. Note that the move is identical if any of the type II mosaic  complement tiles are swapped for type IV tiles.}
\label{fig:bubblerelease}
\end{figure}

We know, however, by Theorem~\ref{thm:burst} that a minimal 
mosaic  complement can never contain a bubble and $T$ is unchanged by a bubble release move so we see immediately the following lemma. 

\begin{lemma}
If $T$ is minimal then $C$ cannot contain tiles on which we can perform a bubble release move.
\label{lem:norelease}
\end{lemma}

The other type IV move is the XX-through-XY move.  Let $C$ contain $e_1$ an XY-edge that has $e_2$ an XX-edge outside of it such that the two edges share a type IV tile.  Switching the tile from IVa to IVb or vice versa will, of course, have no effect on $T$ or $K$, but will replace 
$e_1$ and $e_2$ with a new XY-edge and a new XX-edge.  Call the XY-edge $e_1'$ and the XX-edge  $e_2'$.
The move reduces the overall number of XX-edges outside of XY-edges.  In particular at the very least $e_2'$ is not outside of $e_1'$ and $e_1'$ has fewer XX-edges outside of it than $e_1$ did.

Iterating this process will eventually terminate since the number of type IV tiles is bounded by $|C|$.

We define a knot mosaic together with a mosaic  complement $C$ and associated ordered triple $T$ to be {\em a minimal embedding} for a 
knot $K$ if $T$ is minimal, $C$ contains no loops and in $C$ no XX-through-XY moves are possible.
Corollary~\ref{cor:noloops} together with the process we have just described assures that every knot
$K$ that has a knot mosaic
with $|C| \leq n-4$ has a minimal embedding.  
We call the mosaic  complement in a minimal embedding a {\em minimal mosaic  complement}.

\subsection{Corner-corner moves}

We describe this move in terms of an arc that acts as a local maximum for an edge and is moved downward.  By symmetry,  we can rotate the mosaic any multiple of 90 degrees or reflect along a horizontal or vertical line so the move is equally valid if the arc is a minimum and is moved upward, or one that is concave right and is moved to the right or concave left and is moved left.

\bigskip

{\bf Corner-corner move:}  Let $e$ be an edge in $C$ that intersects 
row $i$ in an arc $a$ that represents a local maximum for $e$.  A local maximum must run directly across row $i, i>2$, from a type IIb
(or IVb) tile
 in  column $s$ to a type IIa (or IVa) tile in column $t$ with $s <t$ as in Figure~\ref{fig:cc}.
We want to reduce the ordered triple $T=(l,l',l'')$ by moving part of the knot up across $a$ and shorten $a$ by moving it down.  Figure~\ref{fig:ccmb} shows the basic move.

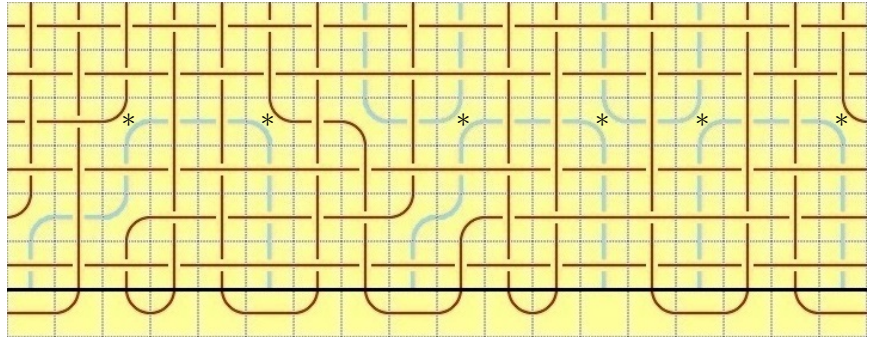
\begin{figure}[htp]
\centering
\centerline{
\begin{tikzpicture}
  \node at (0,0) {\includegraphics[scale=.6]{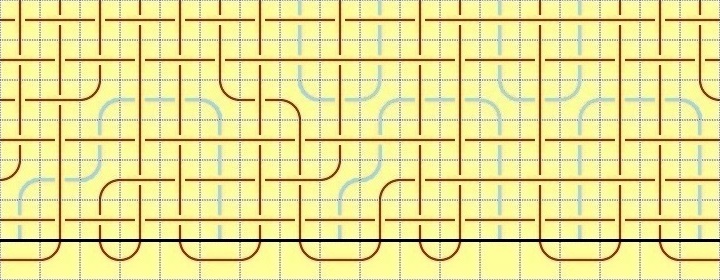}};
  \node at (-4.1,.6) {*};
  \node at (-2.25,.6) {*};
  \node at (.35,.6) {*};
  \node at (2.2,.6) {*};
  \node at (3.53,.6) {*}; 
   \node at (5.38,.6) {*}; 
\end{tikzpicture}
}
\caption{Each * denotes one of four types of corners possible in a 
corner-corner arc (two up to reflective symmetry).}
\label{fig:cc}
\end{figure}

\begin{figure}[htp]
\centering
\centerline{\includegraphics[scale=.7]{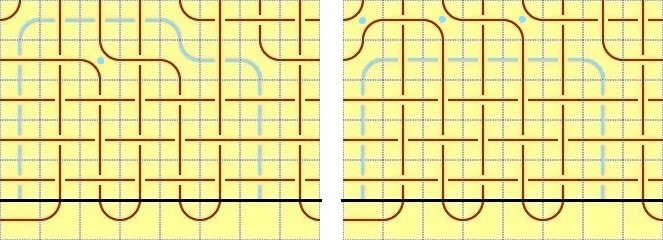}}
\caption{We see a basic corner-corner move.}
\label{fig:ccmb}
\end{figure}

We pay close attention to any portion of the mosaic  complement in tiles $T_{i-1,w}$ with $s \leq w \leq t$ (row $i-1$ directly under $a$). 
Since each of the tiles of $a$ of the form $T_{i,w}$ with $s<w<t$ consists exclusively of type IIIa tiles,  clearly those tiles of the form $T_{i-1,w}$
 cannot ever be type IIc, IId
IIIb, IV or V. 

Certainly $T_{i-1,w}$ can be a Type 0 tile as shown, together with the corresponding corner-corner move in, Figure~\ref{fig:ccmb}. 
On the other hand, if there is a tile in the mosaic  complement $T_{i-1,w}$ with $s<w<t$ 
 that is type IIa, IIb, or  IIIa 
then the corner-corner move is undefined on $a$. Such examples are seen in the nested arcs in Figure~\ref{fig:outermost}.  Another obstruction to the definition we can encounter is that 
if $T_{i-1,t}$ is type IId, IVa, or IVb.  Symmetrically it is also undefined if the mosaic  complement tile $T_{i-1,s}$ is type IIc, IVa, or IVb.
We see arcs of this form in Figure~\ref{fig:mosaic  complementblock2}.
 We will never need to use the corner-corner move in any of the undefined contexts, so the lack of definition here will not be a problem.

\begin{figure}[htp]
\centering
\begin{tikzpicture}
  \node at (0,0) {\includegraphics[scale=.8]{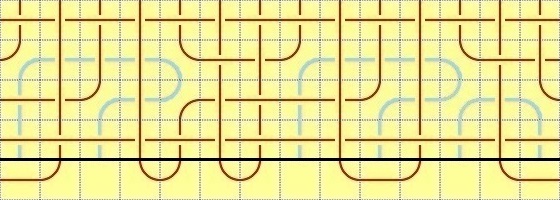}};
  \node at (-3.2,1.1) {$a_{1}$};
  \node at (-1.85,.4) {$a_{1}'$};
  \node at (2.7,1.1) {$a_{2}$};
  \node at (4.1,.4) {$a_{2}'$};
\end{tikzpicture}
\caption{The corner-corner move is not defined on
the arcs $a_1$ and $a_2$ at the top of the edges because of the bottom tile in $a_1'$ and $a_2'$
blocking the move,
but this is not a problem because it is defined on the two-tile arcs $a_1'$ and $a_2'$.}
\label{fig:mosaic  complementblock2}
\end{figure}
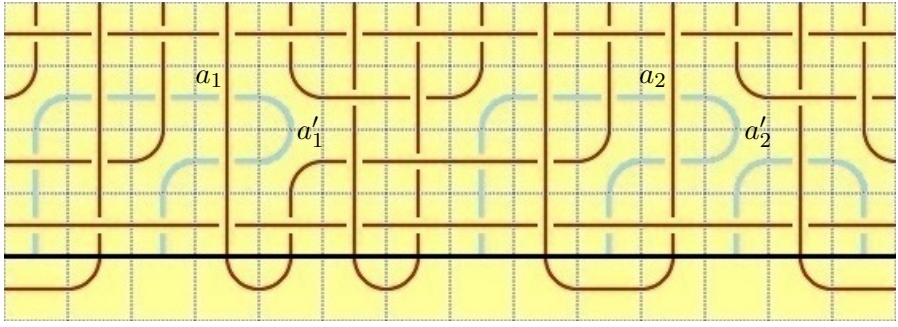

We now focus on the definition of the move in the
situations where it can be applied.  The move at its core just takes an arc $a$ that is a local maximum for
the mosaic  complement and pushes it down one row when there is nothing from the mosaic  complement already below it to block it.  
The exact prescription is given in two parts.  For each $w$, $s<w<t$ we switch 
$T_{i-1,w}$ with $T_{i,w}$.  This tells us how we apply the move to tiles that are between the corners 
of the arc, but not in the corners themselves.
We now specify the move on the two corner tiles and the two tiles directly below them ($T_{i,s}$, $T_{i-1,s}$, $T_{i,t}$ and $T_{i-1,t}$).  The corner tiles 
$T_{i,s}$ and $T_{i,t}$
are either type II tiles or type IV.  When the move is defined tile $T_{i-1,s}$ must be type IIIb tile or type IIc.  On the other corner, 
$T_{i-1,t}$ is either type IIIb or IId. 
As mentioned earlier, the move is not defined if either of the tiles below the corners are type IV; we address this situation later.

The swap for a typical situation is pictured in Figure~\ref{fig:ccmb}.  If the tile $T_{i,s}$ or $T_{i,t}$ is 
type II we replace it with a type 0 tile.  If it is type IVa it is replaced with a type IIc tile.  A type IVb is replaced with type IId.  If the tile $T_{i-1,s}$ is type IIIb it is replaced by a 
type IIb tile. If it is type IId, it is replaced by IIIa.
If $T_{i-1,t}$ is type IIIb then it is replaced by type IIa.  If it is IIc, it is replaced by type IIIa.

\begin{lemma}
A corner-corner move causes a planar isotopy of $K$ and reduces $T$. Hence there cannot be an arc on which a corner-corner move can be applied in a mosaic that minimizes $T$.
\label{lem:nocc}
\end{lemma}

\begin{proof}
The lemma follows directly from the definition of the move.  The tiles between columns $s$ and $t$ swap places, but pairwise remain identical and  thus cannot change $T$.  As seen in the figures no matter which configuration appears in column $s$ and $t$, the ordered triple $T$ decreases in these columns. 
Specifically, the contributions to $|C|$ remains the same, but the contribution in column $s$ to either $|C'|$
or $|C''|$ is reduced by one and the same is true in column $t$. 
\end{proof}

We now turn our attention to the two cases in which the corner-corner move was not defined to see that neither of these is a problem.  The following lemma states that the first one can never occur in a reduced mosaic  complement.

\begin{lemma}
A corner-corner arc $a$ 
with $T_{i-1,t}$ either type IId
or type IVb
or with $T_{i-1,s}$ type IIc or IVa cannot occur if $T$ is minimal.
\label{lemma:switchback}
\end{lemma}

\begin{proof}

Given $a$ in the mosaic  complement running from $T_{i,s}$ to $T_{i,t}$
if $T_{i-1,t}$ in the mosaic  complement is either type IVb  (meaning $R_{i-1,t}$ is type I)
or type IId
then the corner-corner move is not defined on $a$.  
Let the portion of $e$ in tiles $T_{i,t}$ and $T_{i-1,t}$ be called $a'$.  Arcs $a_1'$ and $a_2'$ in Figure~\ref{fig:mosaic  complementblock2} are examples of such arcs.  
If $t=s+1$ and $T_{i,s}$ is type IV then we are in a situation such as Figure~\ref{fig:bubblerelease}, but this
is impossible since $T$ is minimal and the existence of a bubble release move would contradict minimality.  
Given the structure of a corner-corner arc $a$ together with adjacent two-tile corner-corner arc $a'$, this is the only case in which the corner-corner move is not defined on $a'$
Therefore we push it to the left so in all other cases a corner-corner move can be applied to $a'$ reducing $T$
just as it can be to $a_1'$ and $a_2'$ in Figure~\ref{fig:mosaic  complementblock2}.
By Corollary~\ref{lem:nocc} we know that the move cannot
happen if $T$ is minimal, so $T_{i-1,t}$ cannot be type 
IId or IVb. The analogous argument holds by reflective symmetry
if $T_{i-1,s}$ is type IIc or IVa.
\end{proof}

Thus it is not a problem that the corner-corner move was not defined in this context.  We are left only with 
the following situations in which the corner-corner move was not defined. We could have
a corner-corner arc $a$ in the mosaic  complement running across row $i$ from $T_{i,s}$ to $T_{i,t}$ and if $t >s+1$ we have a mosaic  complement tile $T_{i-1,w}$ with $s<w<t$ 
 that is type IIa, IIb, or  IIIa.   If $t=s+1$ then $T_{i,s}$ is type IVb and $T_{i-1,t}$  is type IVa; $T_{i,s}$ may also be type IVb and $T_{i-1,t}$  may be type IVa if $t> s+1$, too, of course.

\begin{lemma}
If $T$ is minimal, then the only way we can have a corner-corner arc $a$ in row $i$ is if 
$a$ is part of a nested series of corner-corner arcs 
$\{a_2, a_3 \dots a_{i-1} \}$ with each $a_j$ contained in row $j$ for $2 \leq j \leq i-1$.
\label{lemma:nested}
\end{lemma}

\begin{proof}
We have seen already that the definition of the mosaic  complement
dramatically limits the choices for tiles beneath $a$ in row $i-1$ so the arc of the mosaic  complement containing 
$T_{i-1,w}$ must be a corner-corner arc $a_{i-1}$ from  $T_{i-1,s'}$ to $T_{i-1,t'}$ for some $s'$ and $t'$ with $s \leq s'<t' \leq t$.
Iterating the process we either find a corner-corner arc that does not have a corner-corner arc below it
in some row $j$ with $2<j \leq i$ contradicting minimality or there are nested corner-corner arcs
extending in every row from $i$ down to $2$ as in the edges $e_1, e_2$ and $e_3$ in Figure~\ref{fig:outermost}.
\end{proof}

\subsection{Corner-edge moves}

We again for simplicity choose to describe this move as it moves an arc $a$ of the mosaic  complement down, but as before, symmetrical moves
to the right, left, or up are all valid by rotations or reflections of the mosaic.  The move is very similar to the corner-corner move as are the arguments about it.

Our goal in applying the corner-edge move is to reduce $T$, 
and we will always do any available corner-corner moves before doing any corner-edge moves, 
so we need not worry about defining the corner-edge move on an edge for which a corner-corner move is possible.

{\bf Corner-edge move:}  Let $e$ be an edge of the mosaic  complement that intersects 
row $i>2$ in an arc $a$ running directly across $i$ in columns 2 through $t$ and turning down in column $t, t \geq 2$.  More precisely $e$ intersects row $i$ in an arc $a$
such that each tile in the mosaic  complement $T_{i,w}$, $w<t$  is a type IIIa tile and tile $T_{i,t}$ is a type
IIa or IVa tile as in Figures~\ref{fig:ceaII} and~\ref{fig:ceaIV}, respectively.

As in the corner-corner move, we pay close attention to any portion of the mosaic  complement in tiles $T_{i-1,w}$ with $2 \leq w<t$ (row $i-1$ directly under $a$). 
Again $T_{i-1,w}$
 cannot ever be type IIc, IId
IIIb, IV or V, but can be a
type 0 tile without causing any problems.

As before, if there is a tile in the mosaic  complement $T_{i-1,w}$ with $s<w<t$ 
 that is type IIa, IIb, or  IIIa 
then the corner-edge move is undefined on $a$. Such examples are seen in the nested edges
$f_1, f_2, f_3$ and $f_4$ in Figure~\ref{fig:outermost}. 

We may have an obstruction where $t=2$, $T_{i-1,t}$ is  type IId
or type IVb and  $T_{i,t} \cup  T_{i-1,t}$ forms a two tile outermost XX-edge, but we never apply a corner-edge move in this context so we do not mind this obstruction.
With this exception we do not encounter an obstruction to the definition where
$T_{i-1,t}$ is either type IId
or type IVa because it would lead to a reduction via a corner-corner move of tiles $T_{i,t} \cup T_{i-1,t}$
to the left which already contradicts the minimality of $T$ for the mosaic  complement.

\begin{figure}[htp]
\centering
\includegraphics[scale=.7]{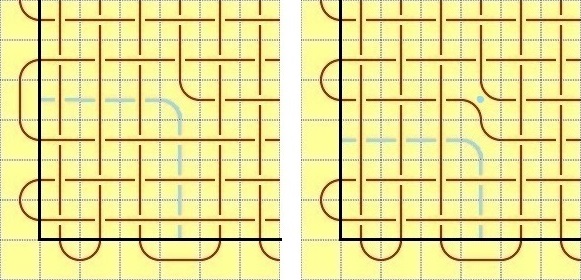}
\caption{The corner-edge move may have a type II tile in its corner.}
\label{fig:ceaII}
\end{figure}

\begin{figure}[htp]
\centering
\includegraphics[scale=.7]{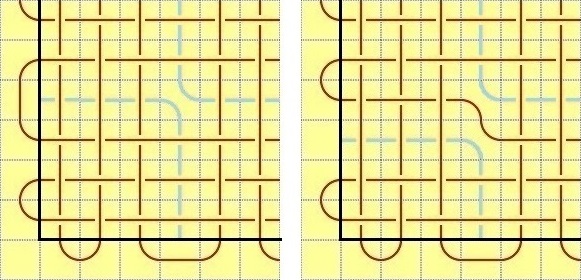}
\caption{Alternatively, the corner-edge move may have a type IV tile in its corner.  The resulting move is only slightly different.}
\label{fig:ceaIV}
\end{figure}

We again reduce $T$ by moving part of the knot across the mosaic  complement.  Figures~\ref{fig:ceaII}, \ref{fig:ceaIV}, and \ref{fig:ceaIIc}
show the basic move.

Because there can be no corner-corner moves in a minimal mosaic  complement and there are never type V tiles in a mosaic  complement, the tile $T_{i-1,t}$ must be either a type IIIb, type IIc or type IVa.
If $t=2$ and  $T_{i-1,t}$ is type IVa then the corner-edge move is not defined, but again this is an obstruction that we do not mind as we will never need to apply it in this context.  Instead examine the case where $t>2$.
If $T_{i-1,t}$ is type IVa we could swap the type IVa mosaic  complement tile with a type IVb tile, replacing mosaic  complement $C$ by mosaic  complement $B$
 without affecting the knot.   Since the only thing we have changed to go from $C$ to $B$ is one type IV tile for another,
 the ordered triple $T_C$ for $C$ is clearly identical to the ordered triple $T_B$ for $B$.
$T_B$ is reduced by a corner-corner move, showing it was not minimal for the knot $K$ whose mosaic  complement is $B$ (and $C$) and therefore $T_C$ also was not minimal for $K$.  Since we always choose our embedding of $K$ so that $T$ is minimal we may assume that tile
$T_{i-1,t}$
 is not type IVa when $t>2$.

We are now left with the possibilities that $T_{i-1,t}$ must be type IIIb (Figure~\ref{fig:ceaII}) or IIc as depicted in Figure~\ref{fig:ceaIIc} and we define the corner-edge move accordingly.
The exact prescription for the move is that for each $w<s$ we switch tile $T_{i-1,w}$ 
with tile $T_{i,w}$.  $T_{i-1,t}$ and 
tile $T_{i,t}$ are treated exactly as they were in the corner-corner move: if $T_{i,t}$ is 
type II we replace it with a type 0 tile.  If it is type IVa it is replaced with a type IIc tile.  
If $T_{i-1,t}$ is type IIIb then it is replaced by type IIa.  If it is IIc, it is replaced by type IIIa.
Typical corner-edge moves are depicted in Figures~\ref{fig:ceaII} through \ref{fig:ceaIIc}.

\begin{figure}[htp]
\centering
\includegraphics[scale=.7]{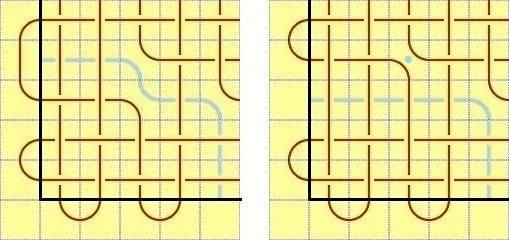}
\caption{$T_{i-1,t}$ may be a type II tile oriented as pictured instead of a type III tile.  The resulting move still reduces $T$.}
\label{fig:ceaIIc}
\end{figure}

Again this move was described in terms of a row, but it can be rotated or reflected to move corner to edge row arcs up and down and corner to edge column arcs right and left.

\begin{lemma}
A corner-edge move causes only a planar isotopy of $K$ and reduces $T$. Therefore in a minimal mosaic, there cannot be an arc on which a corner-edge move may be applied.
\label{lem:noce}
\end{lemma}

\begin{proof}
The lemma follows directly from the definition of the move.  The argument is analogous to Lemma~\ref{lem:nocc}. 
\end{proof}

We now have the moves defined and in the next section will turn our focus to the XX-edges (edges with both endpoints on the same edge), showing that they cannot exist without contradicting minimality.  Then once we know there are no edges of this type we will eliminate XY-edges, too.

\section{Reduction steps towards the main theorem}

\begin{lemma}
\label{lem:XXedgetype2}
If $E = \{e_1,e_2,\dots,e_n\}$ is the set of all edges in $C$ and $|C| \leq n-4$ then there is some $e_i$ containing a type II tile in the mosaic  complement.  If the XX-edges do not share a type IV tile with the XY-edges then at least one of the XX-edges contains a type II tile or the set is empty.  The same is true for the XY-edges.

\end{lemma}

\begin{proof}
Each XX-edge has at least two corners which are either a type II tiles or else Type IV tiles where they meet another edge of the mosaic  complement.  Similarly each XY-edge must have at least one such corner.
To avoid any type II tiles, the mosaic  complement would have to stretch from one side of $S$ to the opposite side,
but this would mean $|C| \geq n-2$ violating the hypothesis of the lemma.
\end{proof}

Lemmas~\ref{lem:XXccup} through \ref{nortorlt} put together will show that if $e$ is an XX-edge with both endpoints on the bottom in a minimal mosaic  complement then $e$ contains exactly one corner-corner arc, and that arc can only be concave down.  
We note that as always, symmetrical arguments can be made by rotation and reflection for edges with endpoints on the other sides of $S$.

\begin{lemma}
\label{lem:XXccup}
If $e$ is an XX-edge in $C$ with both endpoints on the bottom side of $S$ or an XY-edge with one endpoint on the bottom of $S$ and the other on the left side and $|C| \leq n-4$
and $T$ is minimal then 
$e$ cannot contain a corner-corner arc $a$ 
that is concave up.  By symmetry this also means the XY-edge cannot have a corner-corner arc that is concave right.

\end{lemma}

\begin{proof}

By Lemma~\ref{lemma:nested},
$T$ can be reduced via a corner-corner move applied to $a$ unless there is a nested set of corner-corner arcs inside of $a$ including one in each of the rows of $S$ above $a$.  
This, however, cannot happen since it would imply that there are mosaic  complement tiles in every row of $S$, contradicting $|C| \leq n-4$.
\end{proof}

\begin{lemma}
\label{oneccdown}
\label{lem:oneccdown}
Suppose $|C| \leq n-4$ and $T$ minimal. Let $e$ be an XX-edge in $C$ with both endpoints on the bottom edge of $S$
 or an XY-edge with one end point on the bottom of $S$. 
If $a$ is a corner-corner arc of $e$ in row $i$ that is concave down, then $a$ is the only corner-corner
arc on $e$ that is concave down.
\end{lemma}

\begin{proof}

A second concave down corner-corner arc would require
a concave up
 corner-corner arc between the two: an edge in the plane with endpoints at the same height may not have two local maxima
without a local minimum.  We know we cannot have a 
concave up
corner-corner arc on $e$ by Lemma~\ref{lem:XXccup}.
\end{proof}

\begin{lemma}
\label{lem:rtorlf}
Let $e$ be an XX-edge or XY-edge in $C$
where $|C| \leq n-4$ and $T$ is minimal. 
Then $e$ cannot contain a corner-corner arc 
concave to the left (representing a maximum in the direction right) and also a corner-corner arc concave right.

\end{lemma}

\begin{proof}
Each corner-corner arc would need nested corner-corner arcs going all the way to the edge of $S$.  This would, of course, require at least one tile in each column of $S$, contradicting the fact that $|C| \leq n-4$.
\end{proof}

We have now established several lemmas that work for both XX-edges and XY-edges.  The next few lemmas will be just concerned with XX-edges.  After establishing further structure on the XX-edges we will 
be able to return to the XY-edges and deal with them more efficiently.

\begin{lemma}
\label{nort}
Let $e$ be an XX-edge in $C$ with both endpoints on the bottom edge of $S$ with $|C| \leq n-4$ and $T$ minimal.
Let the left endpoint of $e$ be in column $s$ and the right endpoint in column $t$, $s<t$.  We cannot encounter a corner-corner arc concave to the left intersecting column $w$ for $w<t$. The same is true for corner-corner arcs concave right in columns $w$ with $s<w$

\end{lemma}

\begin{proof}
If $a$ is such a corner-corner arc in $e$, to connect 
with $T_{2,t}$ in the first case and $T_{2,s}$ in the second,
$e$ would need to turn around
via a corner-corner arc 
concave in the opposite direction
 contradicting Lemma~\ref{lem:rtorlf}.
\end{proof}


\begin{lemma}
\label{nortorlt}
Let $e$ be an XX-edge in a minimal mosaic  complement $C$ with both endpoints on the bottom edge of $S$ with $|C| \leq n-4$ and $T$ minimal.
Then $e$ cannot have any corner-corner arcs concave to the left or right.

\end{lemma}

\begin{proof}
We show the proof for corner-corner arcs concave to the right
since the proof to the left is identical up to symmetry.
By Lemma~\ref{nort} the concave right corner-corner arc $a$ must start and end in 
column $w$, $w\leq s$.  
However by Lemma~\ref{lemma:nested}, $a$ must have nested corner-corner arcs inside of it extending all the way to the right side
of $S$.  Since
$a$ is on the left side of $e$, this implies that at least 
one of the nested corner-corner arcs for $a$ is also in $e$.  
For $e$ to contain two corner-corner arcs that are concave to the right it must also contain a corner-corner arc concave to the left between them.  This contradicts Lemma~\ref{lem:rtorlf}
.  Thus there were
no corner-corner arcs concave to the right.
\end{proof}

These lemmas imply

\begin{cor} Let $e$ be an XX-edge in minimal mosaic  complement $C$ with both endpoints on the bottom edge of $S$ with $|C| \leq n-4$ and $T$ minimal.
Let the left endpoint of $e$ be in column $s$ and the right endpoint in column $t$, $s<t$.  Let row $i$ contain the maximum
of $e$.  Then $e$ is strictly contained between columns $s$ and $t$ (inclusive) and below row $i$ (inclusive).
\label{cor:contained}
\end{cor}

We now apply this result to outermost XX-edges to build an argument that they must consist of only two tiles exemplified by edge $e_1$ in Figure~\ref{fig:outermost}.

\begin{lemma}
\label{lem:XXedge}
If $e$ is an outermost XX-edge in minimal mosaic  complement $C$ with $|C| \leq n-4$ then 
$e$ consists of two adjacent tiles in the second layer.  Each of the tiles is either type II or type IV.
\end{lemma}

\begin{proof}

Without loss of generality let both endpoints of $e$ be on the bottom edge of $S$.
First we argue that if $e$ is an outermost edge then 
$e$ is totally contained in row 2 and consists of $T_{2,s}$ a type IIb or IVb tile, $T_{2,t}$ a type IIa or IVa tile
and type IIIa tiles $T_{2,w}$ for $s<w<t$ (if $t = s+1$ then this last set of tiles is not used).
Up to rotation, $h$ and $e_1$ are examples of such arcs in Figure~\ref{fig:outermost}.
We then strengthen the result to show that $e$ not only is in the second row, but it contains only two tiles.

Let $a$ be the corner-corner arc of $e$ in row $i$, the highest row that contains a tile of $e$.  If $i=2$ we are done
with the first step of the proof.
If $i>2$ then the row $i+1$ just above $a$  cannot contain any portion of $e$ since it is above the global maximum of $e$ and this dictates that the tiles in row $i+1$ are inside (not outside) of $e$.
In turn this means that the row
$i-1$ just below $a$ must either 
contain part of $e$ or be outside of $e$.  
Since $T$ is minimal, there must be a corner-corner arc $a' \subset C$ below $a$ or we could do a corner-corner move on $a$ pushing it down and reducing $T$.
Then $a'$ cannot be part of an edge other than $e$ since that would imply it was outside of $e$ 
and $e$ is outermost.  Thus $a' \subset e$,
but this contradicts  Lemma~\ref{lem:oneccdown}.  This implies that $a$ is in row 2 and this can only happen if $e=a$ and consists of $T_{2,s}$ a type IIb tile, $T_{2,t}$ a type IIa tile
and $T_{2,w}$ type IIIa tiles for $s<w<t$.

Now we know that $e$ is entirely contained in row 2.
If $e$ contains more than 2 tiles, then the leftmost tile of $e$ can be moved to the right using a corner-edge move to reduce $T$ contradicting minimality.  Thus $e$ must be just 2 tiles long and we are only left with the desired type of outermost XX-edges (see arcs $e_{1},c_{1},c_{2},g_{1}$ and $g_{2}$ in Figure~\ref{fig:outermost}).
\end{proof}

\begin{lemma}
If $e$ is an
XX-edge in a minimal mosaic  complement $C$ with both endpoints on the bottom edge of $S$ and $|C| \leq n-4$ and the maximum of $e$
occurs in column $i$, $i>2$ then there is a set of nested edges 
$\{e_2, e_3 \dots e_{i-1} \}$ outside
of $e$ with the maximum of each $e_j$ in row $j$. 
\label{lem:nestedccarcs}
\end{lemma}

\begin{proof}
We must have nested corner-corner arcs outside of $e$ in each row and each edge has only one corner-corner arc.
\end{proof}

\begin{lemma}
If $e$ is an
XX-edge in minimal mosaic  complement $C$ with both endpoints on the bottom edge of $S$  and $|C| \leq n-4$ and the left endpoint of $e$
occurs in column $s$ at $T_{2,s}$, and $e$ is not outermost, then there is a 
nested set of edges 
$\{e_{s+1}, e_{s+2} \dots e_{k}\}$ outside
of $e$ with the left endpoint of $e_j$ in column $j$ for each 
$j$, $s+1 \leq j \leq k$ 
and $e_k$ an outermost edge in $C$. 
\end{lemma}

\begin{proof}
Because the XX-edges contain no corner-corner arcs that are concave left
or right we can always use a corner-edge move to reduce $T$ unless there is an edge with its endpoint in the adjacent column blocking the move.  
\end{proof}

\begin{lemma}
\label{XXtypeII}
If $e$ is an
XX-edge in $C$ and $|C| \leq n-4$ and all the edges outside of $e$ are nested with each other,
and $e$ contains a type II tile, then $C$ is not minimal.
\end{lemma}

\begin{proof}
Examine the subset of $M_n$ representing $K$.  If the arcs of $K$ contained in $S$ are connected using the tiles 
under $e$ in row 1, as it does in the top mosaic in Figure~\ref{fig:outermostconnect}, then the edge in row 1 also runs under an outermost edge $e'$ outside of $e$ ($e=e'$ if $e$ is outermost).
We can add $e'$ to $K$, connecting it up to the original knot 
as in the bottom picture in Figure~\ref{fig:outermostconnect}, and obtain a knot isotopic to $K$,
but we have reduced the ordered triple contradicting minimality.

If $K$ does not use the tiles in row 1 under $e$, then
$K$ is disjoint from all of the tiles in row 1 between
the endpoints of $e$.  If necessary change crossings between the knot and mosaic  complement -- but not the knot with itself, thus not changing the knot at all -- to make sure $e$ always goes under $K$, and connect $e$ to itself through row $2$ giving a loop.  Remove the loop from the mosaic  complement and add it to the mosaic creating a link mosaic with two components, $K$ and an unknot. This takes us to the middle picture in Figure~\ref{fig:typeIIconnect}.  Then use a corner conversion by placing a type V crossing tile where the
 type II tile of $e$ had been 
going from the middle to the bottom picture in Figure~\ref{fig:typeIIconnect}.  As in the proof of Lemma~\ref{lem:loop2}, a corner conversion takes the connect sum of $K$ with an unknot giving another version of $K$ on an $n \times n$ mosaic, but with a reduced ordered triple contradicting the minimality of $T$.
\end{proof}

\begin{figure}[htp]
\centering
\includegraphics[scale=.6]{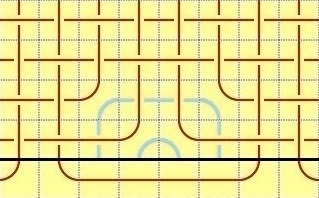}
\hspace{.2in}
\includegraphics[scale=.6]{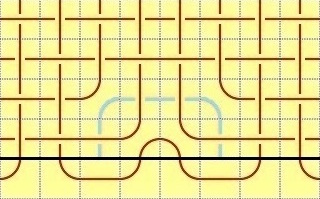}
\caption{If $K$ passes outside of an outermost arc of the mosaic  complement, the move pictured shows the mosaic  complement is not reduced.}
\label{fig:outermostconnect}
\end{figure}

\begin{figure}[htp]
\centering
\includegraphics[scale=.45]{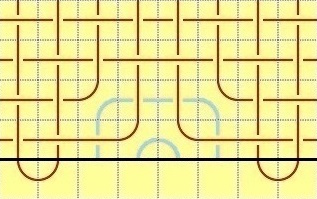}
\hspace{.1in}
\includegraphics[scale=.45]{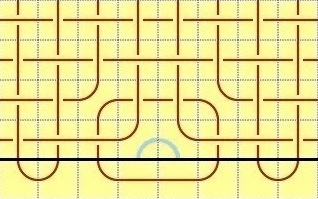}
\hspace{.1in}
\includegraphics[scale=.45]{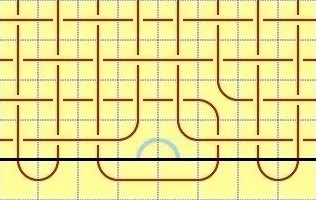}
\caption{If a nested XX-edge in the mosaic  complement has a type II tile in the mosaic  complement
and $K$ does not pass outside the edge, we can alter the mosaic  complement reducing $T$.}
\label{fig:typeIIconnect}
\end{figure}

\begin{lemma}
If $e$ is an
XX-edge in a minimal mosaic  complement $C$ and $|C| \leq n-4$ then all the edges outside of $e$ are nested with each other.
\end{lemma}

\begin{proof}
If not, then examine the outermost edge $e$ which has edges outside of it which are not nested with each other. Let $e_1$ and $e_2$ 
be the innermost non-nested edges outside of $e$
(so $e_1$ is not outside of $e_2$ and vice versa and $e_1$ and $e_2$ are just outside of $e$). Up to rotation, this situation is depicted by arcs $g_{1},g_{2},g_{3}$ and by $c_{1},c_{2},c_{3}$ in Figure~\ref{fig:outermost}. 
None of the corners at the top of $e_1 \cup e_2$ can be type II or $T$ is not minimal by the previous lemma, but if they are all type IV tiles then $e$ has a corner-corner arc that is concave up contradicting Lemma~\ref{lem:XXccup}.
\end{proof}

\begin{figure}[htp]
\centering
\begin{tikzpicture}
  \node at (0,0) {\includegraphics[scale=.7]{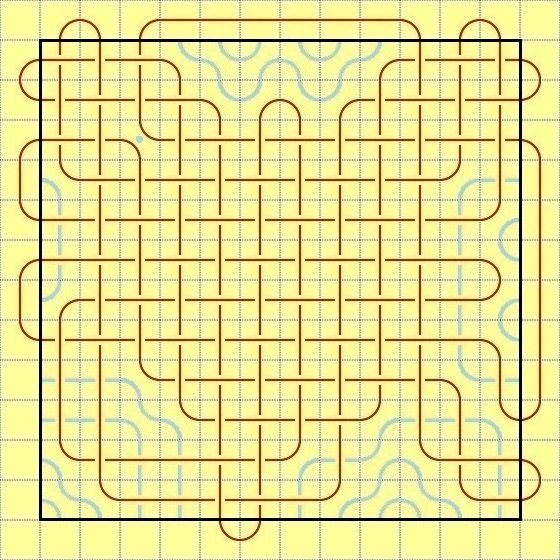}};
  \node at (-4.05,-4.8) {$f_{1}$};
  \node at (-3.32,-4.8) {$f_{2}$};
  \node at (-2.6,-4.8) {$f_{3}$};
  \node at (-1.85,-4.8) {$f_{4}$};
  \node at (2.7,-4.8) {$e_{1}$};
  \node at (3.4,-4.8) {$e_{2}$};
  \node at (4.1,-4.8) {$e_{3}$};
  \node at (-4.65,-.38) {$h$};
  \node at (-1.85,4.6) {$c_{1}$};
  \node at (-1.12,4.6) {$c_{2}$};
  \node at (.43,4.6) {$c_{3}$};
  \node at (4.67,1.85) {$g_{1}$};
  \node at (4.67,.35) {$g_{2}$};
  \node at (4.67,-1.15) {$g_{3}$};
\end{tikzpicture}
\caption{An example of a knot mosaic and its mosaic  complement.}
\label{fig:outermost}
\end{figure}
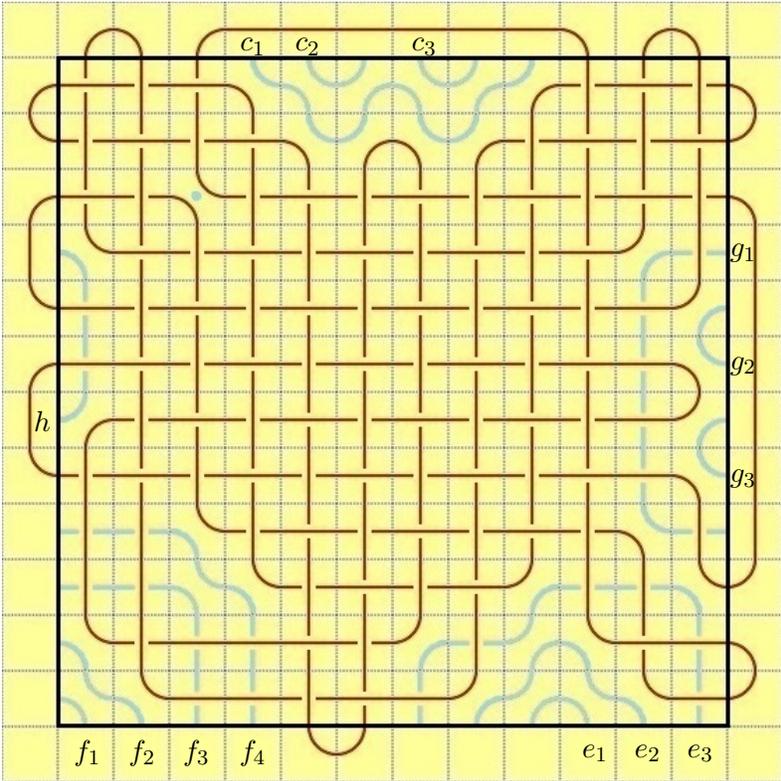

These lemmas imply

\begin{thm}
$C$ contains no XX-edges containing a type II tile if $C$ is chosen minimally.
\label{thm:noxx}
\end{thm}

Now that we know that all the corners in an XX-edge are type IV we exploit this fact to get rid of all XX-edges.

\begin{theorem} If $|C| \leq n-4$ and
$C$ is chosen minimally, then $C$ contains no XX-edges.
\label{thm:noXXedges}
\end{theorem}

\begin{proof}

If $C$ does contain an XX-edge then there is at least one
that is not outside of any of the other XX-edges -- the innermost edge from any of the nested sequences would suffice.  Call that edge $e$.
Without loss of generality let $e$ have both end points on the bottom of $S$, specifically in tiles $T_{2,s}$ and $T_{2,t}$ with $s<t$.
We know $e$ has no type II tiles by Theorem~\ref{thm:noxx}, but it must contain two type IV tiles at its maximum.  Let $f$ be an edge that meets $e$ in one of these type IV tiles. Note that the tile coming from a maximum for $e$ implies that $f$ is not outside of $e$.  Since $e$ has endpoints on the bottom of $S$
and $|C| \leq n-4$ we know $f$ cannot have an end point on the top edge of $S$.  If there is a second edge $g$ that shares the other type IV tile from the maximum of $e$, it cannot be the case that one of these edges had an end point on the right side of $S$ and the other on the left side, since such an edge would stretch across $S$, forcing $|C|\geq n - 2$. 
So without loss of generality we may assume that the end points of $f$ and $g$ are contained in at most the left and bottom sides of $S$.  

We next argue that both end points of $f$ (or $g$)
cannot just be on the bottom of $S$.  If $f$ has both end points on the bottom it is by definition an XX-edge, but recall that 
we chose $e$ so it was not outside of any XX-edges so we know that $e$ is not outside of $f$.  As a result both end points of $f$ have to be contained in rows that are on the same side of rows containing the end points of $e$.  In 
particular $f$ must have end points in tiles $T_{2,u}$ and $T_{2,v}$ with $u<v$ since they are on the bottom of $S$, but then since the two edges are not nested and do not intersect we must have $u<v<s<t$ or $s<t<u<v$.  Either way 
by Corollary~\ref{cor:contained}
$f$ never intersects any of the columns between $s$ and $t$ and $e$ never leaves these columns so they  cannot contain a common type IV tile.  This implies that $f$ does not have both end points on the bottom of $S$.  

Now without loss of generality $f$ (and any edge sharing a type IV tile with $e$) either has both end points on the left side of $S$ if it is an XX-edge or one on the left side of $S$ and the other on the bottom side if it is an XY-edge. In the latter case, since $e$ is not outside of $f$ we also know that the end point on the bottom of $S$ is in some tile $T_{2,u}$ with $u < s < t$.   

Examine again the top right type IV tile from the maximum of $e$ and the edge $f$ that shares this tile. The portion of $f$ within this tile looks like a IIc tile.  Because both of its end points are to the left of $e$ and it does not intersect $e$ we know $f$ 
must contain a concave up corner-corner arc to
contain this tile.
This, however, contradicts Lemma~\ref{lem:XXccup}.

Thus $e$ cannot exist so there are no XX-edges in the mosaic  complement.
\end{proof}

From here we proceed by eliminating XY-edges.

\begin{theorem} If $|C| \leq n-4$ and
$C$ is chosen minimally, then $C$ contains no XY-edges.
\end{theorem}

\begin{proof}

We know all edges must be XY-edges, and that if the set of edges is nonempty then at least one of the XY-edges contains a type II tile by Lemma~\ref{lem:XXedgetype2}.

Let $e$ be such an XY-edge.  We can rotate the entire mosaic if necessary until it encloses the bottom left corner, so let $e$ have one endpoint in tile $T_{s,2}$ on the left edge of $S$ and the other in
 $T_{2,t}$ on the bottom edge of $S$ enclosing the bottom left corner on its outside.  
First observe that $e$ cannot contain a corner-corner arc. 
By Lemma~\ref{lem:nestedccarcs}, any corner-corner arc requires a set of nested corner-corner arcs terminating in an XX-edge on the boundary of $S$.  However by Theorem~\ref{thm:noXXedges}, $C$ contains no XX-edges so this is impossible.

The lack of any corner-corner arcs implies that $e$ is completely contained in the rows below row $s$ (inclusive) and to the left of column $t$ (inclusive).
In turn this implies that if
$s > 2$ then there must be a set of nested XY-edges $\{e_2, e_3 \dots e_w \dots e_{s-1} \}$ for each $w$, $2 \leq w < s$.  If not we could apply a corner-edge move to reduce $T$.  The analogous result holds for the endpoints on the bottom of $S$ with respect to the columns.  

Now we close the argument in the same manner as before.  The outermost XY-edge must be a type II tile or one of the arcs from a type IV tile in $T_{2,2}$. The next most outermost has one endpoint in tile $T_{2,3}$ and the other in $T_{3,2}$ etc.  
Examine the tiles representing $K$.
If $K$ runs through row 1 and column 1 past the ends of these arcs, then we add the outermost XY-edge to $K$ and connect it up to give a planar isotopy of $K$ and reducing $T$, contradicting minimality.  If not, recall that $e$ contains a type II tile. 
Since $K$ does not go through the tiles in row 1 or column 1 outside of $e$
we can turn $e$ into an unknot by
hooking the endpoints of $e$ to itself through these tiles.  
There is no effect on $K$ if we assume that $e$ always passes under $K$.  As in previous theorems 
we use a corner conversion to replace the hypothesized type II tile from $e$ 
with a type V tile to connect sum the new unknot with $K$, yielding another embedding of a knot isotopic to $K$ for which $T$ has decreased contradicting minimality.  Thus there can be no XY-edges.
\end{proof}

Since the mosaic  complement contains no XX-edges, no XY-edges and no loops, we may now conclude the following Corollary.

\begin{cor}
If $M_n=M_{2k}$ is an even knot mosaic yielding knot $K$, with minimal mosaic  complement $C$ and $|C| \leq n-4$,
 then $l'=|C'|=0$. 
Therefore we may assume $C$ consists exclusively of type 0 tiles.
 \label{cor:type0}
\end{cor}

\section{Proof of the main result for even boards}

We use Corollary~\ref{cor:type0} to show our main theorem below.

\begin{theorem}
If $M_n=M_{2k}$ is an even knot mosaic yielding knot $K$, then the crossing number of $K$ is less than $(n-2)^2 - (n-4)$.

\label{thm:evenub}
\end{theorem}

\begin{proof}

A mosaic with $|C|=l$ and 
$C'= \emptyset$ has $l$ type 0 tiles in the mosaic  complement and nothing else.  It therefore 
is obtained from a saturated mosaic by
smoothing $l$ crossings.  In 
the language of tiles, we are
replacing $l$ type V tiles in the link with $l$ type IV tiles.
Each time this is done the number of components in the mosaic changes by at most one.

If $l > n-4$ then $K$ has at most $(n-2)^2-(n-5)$ crossings failing to exceed our bound.  
As we saw in the section on saturated mosaics, a
saturated even mosaic has $n-2$ or $n-3$ components depending on how the arcs contained in $S$ are connected in the boundary tiles of the mosaic.  
Since we are only smoothing $l$ crossings, we see that if $l < n-4$ this leaves at least 2 components and we did not really have a knot mosaic. 
In this context we insist that we are left with a knot,
that $l \leq n-4$ and also that $l \geq n-4$ so it must be that $l=n-4$.

If the saturated mosaic starts with $n-2$ components,
smoothing $n-4$ crossings 
still leaves at least 2 components, so to avoid a contradiction the original saturated mosaic must have been connected to yield
 $n-3$ components.  This, however, can only happen when we also have nugatory crossings in each of the four corners of the mosaic.  
If the $n-4$ type 0 tiles yield a knot then 
none of them are in a corner of $S$ as 
smoothing one of these crossings
fails to lower the number of components in the link.
This means that the knot that results from smoothing $n-4$ crossings
will have $(n-2)^2-(n-4)$ crossings, but it also still has the  
4 trivial loops in the corners that can be removed with type I Reidemeister moves.
Thus $K$ could be embedded with 4 fewer crossings, showing its crossing number is at most $(n-2)^2 - (n-4)-4 < (n-2)^2 - (n-4)$ so this knot does not exceed our bound on crossing number on even mosaics.    Therefore a knot mosaic on an even board cannot have crossing number greater than or equal to $(n-2)^2 - (n-4)$. 
\end{proof}

 Note that the trefoil establishes that this bound is sharp since it is achievable on $M_4$ showing a knot of crossing number 3 can be built on a $4 \times 4$ board, but our bound says that we cannot have a knot of crossing number 4 on such a board.

%

The primary goal of this paper is to refine existing upper bounds for crossing number. Theorems~\ref{thm:oddub} and~\ref{thm:evenub} together establish the following upper bound for crossing number given mosaic number.

\begin{thm}[New Upper Bound for Crossing Number]
\label{thm:newub}
Given an $m$-mosaic and any knot $K$ that is projected onto the mosaic, the crossing number $c$ of $K$ is bounded above by the following:\\
\begin{equation*}
  c \leq
  \begin{cases}
    (m - 2)^{2} - 2 & \quad \text{if $m = 2k + 1$}\\
    (m - 2)^{2} - (m - 3) & \quad \text{if $m = 2k$.}
  \end{cases}
\end{equation*}
\end{thm}

\section{Lower bound for mosaic number}

At the beginning of this paper we used Theorem \ref{thm:lowerbound} to relate crossing number and mosaic number. In a similar fashion, Theorem \ref{thm:newub} may be used to bound a knot's mosaic number from below. First we define 
\begin{align*}
  B_{1} &= \sqrt{2 + c} + 2\\
  B_{2} &= \frac{5 + \sqrt{4c - 3}}{2}
\end{align*}
As a corollary to Theorem~\ref{thm:newub}, we have

\begin{cor}[New Lower Bound for Mosaic Number]
\label{cor:newlb}
Let $K$ be a knot with crossing number $c$ and mosaic number $m$. Then $m \geq \min\{B_{1},B_{2}\}$.
\end{cor}

This will prove useful in future computations of mosaic number. For now, we will briefly explore the behavior of $B_{1}$ and $B_{2}$. It is easy to see that $B_{1}$ and $B_{2}$ are asymptotic, as
$$
\lim_{c\rightarrow\infty} \frac{B_{1}}{B_{2}} = 1.
$$
Informally, this means that $B_{1}$ and $B_{2}$, as functions of $c$, grow at relatively the same rate. A stronger result is that the difference $|B_{1} - B_{2}|$ is bounded by $\frac{1}{2}$, and although this difference is always increasing, it turns out that
$$
\lim_{c\rightarrow\infty} |B_{1} - B_{2}| = \tfrac{1}{2}.
$$
This result may be somewhat surprising: our work has shown that the even and odd cases require different approaches, but in reality the estimates for each are actually quite similar and the predictive power of one never strays too far from the other.

\section{Future research directions}

Our research has provoked several questions about knot mosaics which are left open to further investigation.

\begin{question}
How does mosaic number behave for the connect sum of knots?
\end{question}

\begin{question}
Are there any knots whose mosaic number is 2 greater than the number predicted by Theorem \ref{thm:newub}?
\end{question}

\begin{question}
What are the mosaic numbers for all knots of 10 crossings or fewer? 
\end{question}

We note that  Lee, Ludwig, Paat, and Peiffer compute the mosaic
number of all prime knots of 8 crossings or fewer \cite{LLPP} so progress is currently being made in this direction.

It is also natural to look at a more general class of mosaics where instead of insisting the board be $n \times n$ we allow it to be $n \times m$.  One might then define the rectangular mosaic number in terms of the number of tiles in the mosaic or perhaps even better the number of tiles on its interior.  Mosaics that need not be square should allow for more efficient embeddings especially in the case of knots that are not prime.
 
\begin{question}
How does crossing number relate to rectangular mosaic number?
\end{question}

The authors in \cite{lee} establish an upper bound on the mosaic number of a knot using arc index, and use this to prove stronger bounds for several classes of knots. Their results, together with Theorem \ref{thm:newub} and Corollary \ref{cor:newlb} of this paper, provide a clearer picture of the relation between the mosaic number and crossing number of a knot. However, in general a more precise formula that relates the mosaic number and crossing number of a knot would be desirable.  The mosaic number of some relatively simple knots remains unknown including many prime knots with 9 crossings and some composite knots with fewer than 9 crossings.


Note that a $5 \times 5$ board is simple enough that it can only support 9 crossings.  It is likely that one could build these composite knots on a $6 \times 6$ board and then simply analyze all possible $5 \times 5$ boards to establish the results for some of these composite knots, but for knots with higher crossing numbers the complexity of their mosaic representations increases rapidly making a case by case analysis impractical.

\bigskip

Finally, we suggest an extension of knot mosaics to three dimensions. Using cubic blocks rather than square tiles, we can represent a knot in its three dimensional form without imposing crossings onto a
two dimensional representation, while still maintaining a degree of rigidity. Define a \emph{standard cube} as an analog of a mosaic tile: a cube contains 0, 1, 2 or 3 strands, and each face of the cube intersects at most 1 strand in the center of the face. In addition, for each strand within a cube there is at most one critical point in any direction on the interior arc of the strand. Define an \emph{n-cubic knot} as an $n\times n\times n$ array of suitably connected standard cubes. Furthermore, define the \emph{grid number} of a knot (analog of mosaic number) to be the smallest natural number $g$ such that the knot is representable as a $g$-cubic knot. Note that mosaic number is a (bad) upper bound for grid number. Further research questions on the topic of cubic knots can be asked such as the one proposed below.

\begin{question}

For a knot $K$ with mosaic number $m$ and grid number $g$, $g\leq m$. Find a sharper upper bound for $g$.
\end{question}

The authors would like to thank Sam Lomonaco and Lou Kauffman for use of many of the figures in this paper, as well as Joe Paat and Lew Ludwig for inspirational discussions resulting from their paper with Erica Evans \cite{lep}.







\end{document}